\documentclass[11pt, a4paper]{amsart}
\usepackage{amssymb,amsmath}
\usepackage{amsthm, dsfont, a4}
\usepackage{hyperref, xcolor}

\usepackage{tikz}
\usetikzlibrary{arrows}
\usepackage[active]{srcltx}
\usepackage[short,nodayofweek]{datetime}

\definecolor{my-blue}{cmyk}{1,0.6,0,0}

\definecolor{my-green}{cmyk}{0.8,0,1,0.5}

\hypersetup{
colorlinks=true, 
linkcolor=my-blue, 
citecolor=my-green,
urlcolor=black
}

\def\FF{{\mathbb F}}
\def\NN{{\mathbb N}}
\def\ZZ{{\mathbb Z}}
\def\QQ{{\mathbb Q}}
\def\id{{\rm id}}

\newcommand{\ls}[1]{(\!(#1)\!)}  
\def\betr#1{\lvert #1\rvert}

\DeclareMathOperator{\GL}{GL}

\newcommand{\Kbar}{\overline{K}_\infty}
\newcommand{\Fq}{\mathbb{F}_{\! q}}

\newcommand{\KI}{K_\infty}
\newcommand{\e}{\exp_\phi}

\theoremstyle{plain}
\newtheorem{thm}{Theorem}[section]

\newtheorem{prop}[thm]{Proposition}

\theoremstyle{definition}

\newtheorem{exmp}[thm]{Example}
\newtheorem{rem}[thm]{Remark}

\begin{document}

\title{On field extensions given by periods of Drinfeld modules}
\author{Andreas Maurischat}
\date{}



\begin{abstract}
In this short note, we answer a question raised by M.~Papikian on a universal upper bound for the degree of the extension of $K_\infty$ given by adjoining the periods of a Drinfeld module of rank $2$. We show that contrary to the rank $1$ case such a universal upper bound does not exist, and the proof generalises to higher rank. Moreover, we give an upper and lower bound for the extension degree depending on the valuations of the defining coefficients of the Drinfeld module. In particular, the lower bound shows the non-existence of a universal upper bound.
\end{abstract}

\maketitle

\section{Introduction}


Let $K=\Fq(\theta)$ be the rational function field over the finite field $\Fq$ with $q$ elements. A Drinfeld module over $K$ of rank $r\geq 1$ (and of generic characteristic) is given by a homomorphism of $\Fq$-algebras
\[\phi:\Fq[\theta] \to K\{\tau\}, \theta\mapsto \phi_{\theta}=\theta+a_1\tau+\ldots+a_r\tau^r \] 
for some $a_1,\ldots, a_r\in K$, $a_r\neq 0$, 
where $K\{\tau\}$ is the skew polynomial ring of $\Fq$-linear maps $K\to K$. More precisely,
$\tau:K\to K$ is the $q$-power Frobenius map on $K$, and 
\[  \bigl(\sum_{i=0}^n a_i\tau^i\bigr)\cdot \bigl(\sum_{j=0}^m b_j\tau^j\bigr)=\sum_{i=0}^n \sum_{j=0}^m a_ib_j^{q^i}\tau^{i+j}, \]
for all $\sum_{i=0}^n a_i\tau^i,\sum_{j=0}^m b_j\tau^j\in K\{\tau\}$.\footnote{We will only provide those properties of Drinfeld modules that we will use in this note.
For more details on Drinfeld modules we refer the reader to the standard text books, e.g.~\cite{dg:bsffa}, or \cite{dt:ffa}.}

Let $\Kbar$ denote the algebraic closure of 
$\KI=\Fq\ls{\tfrac{1}{\theta}}$, the completion of $K$ at the infinite place.
Attached to $\phi$ is the so called exponential function $\e:\Kbar\to \Kbar, x\mapsto \sum_{i=0}^\infty \alpha_i x^{q^i}$ which is the unique $\Fq$-linear map satisfying $\e(a\cdot x)=\phi_a(\e(x))$ for all $a\in \Fq[\theta]$ and $x\in \Kbar$, as well as $\frac{\partial}{\partial x} \e=\id_{\Kbar}$ (i.e.~$\alpha_0=1$). 

It is well known that $\e$ is surjective, and that its kernel $\Lambda_\phi$ is a discrete $\Fq[\theta]$-submodule of $\Kbar$ -- called the period lattice of $\phi$ -- and its rank equals the rank $r$ of the Drinfeld module, i.e. the degree in $\tau$ of $\phi_{\theta}$. Hence, one has an isomorphism of $\Fq[\theta]$-modules $\Kbar/\Lambda_\phi \to \Kbar$ with the scalar action on the quotient $\Kbar/\Lambda_\phi$, and the action via $\phi$ on the domain $\Kbar$.

Drinfeld's uniformization theorem over $\Kbar$ states that $\phi \leftrightarrow \Lambda_\phi$ provides a bijective correspondence between Drinfeld modules over $\Kbar$ and discrete finitely generated $\Fq[\theta]$-submodules of $\Kbar$, which explains one aspect why Drinfeld modules are considered as analogues of elliptic curves.
For further analogies, we refer the reader to the survey article of Deligne-Husemoller \cite{pd-dh:sdm}, or even the original papers by Drinfeld \cite{vd:em}, \cite{vd:em2}.

Assume now, we are given a Drinfeld module $\phi$ over $K$ of rank $r$, and let the lattice $\Lambda_\phi$ be generated by $z_1,\ldots, z_r\in \Kbar$. Then $\KI(z_1,\ldots, z_r)/\KI$ is a finite extension.
There are several questions on those extensions $\KI(z_1,\ldots, z_r)/\KI$ that naturally occur:
\begin{enumerate}
\item Is there an upper bound on the degree $[\KI(z_1,\ldots, z_r):\KI]$ independent of $\phi$?
\item Is there such a global bound for fixed rank $r$?
\item Is there a global bound on the degree of the extension of the constant fields, i.e. on the degree $[\overline{\Fq}\cap \KI(z_1,\ldots, z_r) : \Fq]$?
\item Is there such a global bound on the degree of the extension of the constant fields  for fixed $r$?
\end{enumerate}

Of course, a positive answer for the first question would imply positive answers for the other questions as well, and also a positive answer on questions (2) or (3) would imply one for question (4).

However, it is easy to construct Drinfeld modules of increasing rank, where even the constant field extensions get arbitrarily large (see Example \ref{ex:arbitrary-large}).
So questions (3) and (1) have a negative answer.

The questions that we were asked by M.~Papikian, were therefore questions (2) and (4).

For rank $r=1$, $\Lambda_\phi=\Fq[\theta]\cdot z$, it is known that the degree $[\KI(z):\KI]$ is bounded by $(q-1)$. Namely, in this case, the extension $\KI(z)$ equals $\KI(e_0)$ where $e_0\ne 0$ is a $\theta$-torsion element of $\phi$, i.e.~satisfies $\phi_\theta(e_0)=0$, and the equation $e_0^{-1}\phi_\theta(e_0)=0$ provides a polynomial relation for $e_0$ over $K$ of degree $q-1$.
So question (2) and (4) have a positive answer for $r=1$, and Gekeler even gave an explicit formula for the degree in this case (see \cite[Thm.~4.11]{eg:gitcm}).

\medskip

The main result of this note is that question (2) has a negative answer for $r=2$, and the proof easily generalizes to higher rank. In the considered rank $2$, however, we give more precise bounds on the extension depending on the valuations at infinity of the coefficients $a_1$ and $a_2$ (where as above $\phi_\theta=\theta+a_1\tau+a_2\tau^2$). In Theorem \ref{thm:bounds}, we do not only give a lower bound for the extension which answers question (2) in the negative (additional statement in part b)), but also give upper bounds which might be useful for computational aspects.

This still leaves open the answer to question (4), and unfortunately, we are not able to solve it. A positive answer to question (4) would even be more interesting, since it would imply that the moduli space of Drinfeld modules over $K$ of rank $r$ is geometric over some finite extension of $\Fq$.

\section{Notation}

Let $K=\Fq(\theta)$ be the rational function field over the finite field $\Fq$ with $q$ elements, and $A=\FF_q[\theta]$ the polynomial ring inside $K$. Let
$\KI=\FF_q\ls{\tfrac{1}{\theta}}$ be the completion of $K$ at the infinite place, and $\Kbar$ an algebraic closure of $\KI$.
On $K_\infty$ we take the $\infty$-adic valuation $v:\KI\to \ZZ\cup \{\infty\}$ given by
$v(\theta)=-1$, and extend it to a $\QQ$-valued valuation on the algebraic closure $\Kbar$. The reader should have in mind the associated absolute value $\betr{\cdot}$ given by
$\betr{x}=q^{-v(x)}$ for all $x\in \Kbar$, when we speak of ``convergent series'', ``small neighbourhoods'' etc. However, we will not explicitly use $\betr{\cdot}$, but only the valuation $v$.

As in the introduction, $\tau:\Kbar\to \Kbar$ denotes the $q$-power Frobenius, and
\[\phi:A \to K\{\tau\}, \theta\mapsto \phi_{\theta}=\theta+a_1\tau+\ldots+a_r\tau^r \] 
with $a_1,\ldots, a_r\in K, a_r\neq 0$, a Drinfeld module over $K$ of rank $r$. The associated exponential map is denoted by $\e$, and the period lattice by $\Lambda_\phi$. The exponential map $\e$ has a local inverse $\log_\phi$, i.e.~there is some neighbourhood $B\subseteq \Kbar$ of $0$, and a map $\log_\phi:B\to \Kbar$ such that $\e\circ \log_\phi=\id_B$. Both $\e$, and $\log_\phi$ have power series expansions around $0$ with coefficients in $K$, the one for $\e$ converging on all of
$\Kbar$, the one for $\log_\phi$ converging on $B$.

 For $a\in A$, let the $a$-torsion points of $\phi$ be denoted by
\[ \phi[a]:=\{ x\in \Kbar \mid \phi_a(x)=0 \}. \]
This is a free $A/(a)$-module of rank $r$, and hence an $\Fq$-vector space of dimension $r\cdot \deg(a)$.
Later we will only consider $\theta$-powers, i.e. $a=\theta^{n+1}$ for $n\geq 0$, and mention already that for $n\geq 1$,
\begin{equation}\label{eq:torsion-tower} \phi[\theta^{n+1}] \,\,=\,\, \{ x\in \Kbar \mid \phi_\theta(x)\in \phi[\theta^{n}] \} 
\,\,=\,\, \bigoplus_{j=1}^r \Fq e_{n,j} \oplus \phi[\theta^{n}] 
\end{equation}
where $e_{n,1},\ldots, e_{n,r}$ are any elements in $\phi[\theta^{n+1}]$ that are linearly independent modulo
$\phi[\theta^{n}]$. 

For getting the bounds mentioned in the introduction, we first state a connection between the lattice $\Lambda_\phi$ and the $\theta$-power torsion. As we couldn't find a reference for this connection, we also give its proof.

\begin{prop}\label{prop:lattice--torsion-points}
For any Drinfeld module $\phi$ defined over $K$ with period lattice $\Lambda_\phi$, one has
\[  \KI(\Lambda_\phi)=\bigcup_{n\in\NN} \KI(\phi[\theta^{n+1}]). \]
\end{prop}

\begin{proof}[Sketch of proof]
The exponential $\e:\Kbar\to \Kbar$ and its local inverse $\log_\phi:B\subset \Kbar\to \Kbar$ are given by power series with $K$-coefficients. Since finite extensions of $\KI$ are still complete, for any $x\in \Kbar$ one has $\e(x)\in \KI(x)$, and for all
$y\in B\subset \Kbar$ one has $\log_\phi(y)\in \KI(y)$.

As $\e$ is surjective, every $\theta^n$-torsion element $e$ is given as 
\[  e=\e\left( \theta^{-n}\lambda \right) \]
for some $\lambda\in \Lambda_\phi$. Hence, we immediately obtain:
\[  \KI(\Lambda_\phi)\supseteq \bigcup_{n\in\NN} \KI(\phi[\theta^n]). \]
On the other hand, for every $\lambda\in \Lambda_\phi$, there is some $n\in \NN$ such that
\[  e:=\e\left( \theta^{-n}\lambda \right) \]
lies in the radius of convergence $B$ of $\log_\phi$, and we get $\lambda$ as
\[  \lambda =\phi_\theta^n\log_\phi(e)\in \KI(\phi[\theta^n]) \]
showing the reverse inclusion.
\end{proof}

\begin{exmp}\label{ex:arbitrary-large}
Using the previous proposition, it is easy to construct Drinfeld modules where the extension in question has arbitrarily large degree, and even the extension of constants is arbitrarily large. Namely for any $r\geq 1$, take $\phi_\theta=\theta - \theta \tau^r$. Then $\theta$-torsion is given by the roots of
$\phi_\theta(X)=\theta X - \theta X^{q^r}=-\theta(X^{q^r}-X)$. These are just the elements of the field $\mathbb{F}_{\!q^r}$, and hence also the extensions in question contain $\mathbb{F}_{\!q^r}$, and the extension degree is at least $r$.
\end{exmp}

\section{Bounds on the extension by lattice points}

Proposition~\ref{prop:lattice--torsion-points} allows us to study the field extension generated by the $\theta$-power torsion points, instead of working with the lattice directly.

For studying the torsion extensions, we intensively use the Newton polygons for the defining equation
of a $\theta^{i+1}$-torsion point $e_i$ given by $\phi_{\theta}(e_i)=e_{i-1}$. Here $e_{i-1}$ is a previously ''determined'' $\theta^{i}$-torsion point. Furthermore, using the description \eqref{eq:torsion-tower}, we see that the extension by all $\theta$-power 
torsion is determined by a basis of $\theta$-torsion and one \textit{convergent division tower} for each basis element $e_0$, i.e.~a sequence $(e_i)_{i\geq 1}$ with $\phi_{\theta}(e_i)=e_{i-1}$ and 
$\lim_{i\to \infty} v(e_i)=\infty$.

\begin{thm}\label{thm:bounds}
Let $\phi:\Fq[\theta]\to K\{\tau\}$ with $\phi_\theta=\theta+a_1\tau+a_2\tau^2$ ($a_1,a_2\in K$, $a_2\neq 0$) be a Drinfeld module of rank $2$.
\begin{enumerate}
\item[a)] If $v(a_1)\geq\frac{v(a_2)-q}{q+1}$, then $\KI(\Lambda_\phi)=\KI(\phi[\theta])$ and $[\KI(\Lambda_\phi):\KI]$ divides $(q^2-1)\cdot (q^2-q)$.
\item[b)] If $v(a_1)<\frac{v(a_2)-q}{q+1}$, let 
$n=\max \{ j\in \NN_0\mid v(a_1)<\frac{v(a_2)-q^{j+1}}{q+1} \}$.
Then 
\[  \KI(\Lambda_\phi)=\bigcup_{j=0}^n \KI(\phi[\theta^{j+1}]), \]
and $[\KI(\Lambda_\phi):\KI]$ divides $(q-1)^2\cdot q^{n+1}$.
If furthermore $v(a_1)-v(a_2)$ is prime to $q$, then $q^{n+1}$ divides
$[\KI(\Lambda_\phi):\KI]$.
\end{enumerate}
\end{thm}

\begin{proof}
The non-zero elements of $\theta$-torsion are the roots of the polynomial 
$\phi_\theta(X)/X=a_2X^{q^2-1}+a_1X^{q-1}+\theta $.
The line through the points $(0,v(\theta))=(0,-1)$ and $(q^2-1,v(a_2))$ is given by the equation
$y=-1+\frac{v(a_2)+1}{q^2-1}\cdot x$.

\noindent
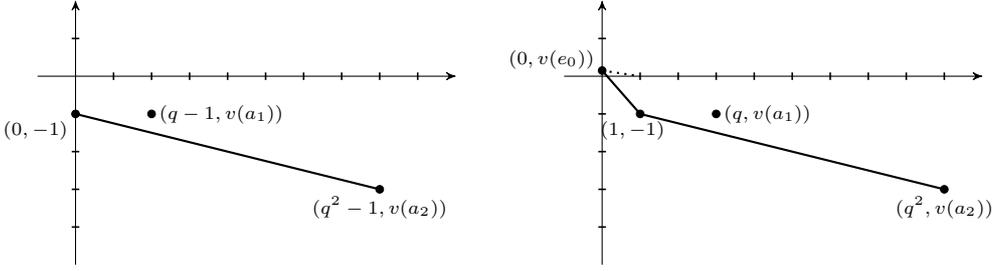
\begin{figure}[ht] 
\begin{tikzpicture}[
    scale=0.5,
    axis/.style={thin, ->, >=stealth'},
    every node/.style={color=black},
    ]
    \tiny

    \draw[axis] (-1,0)  -- (10,0) ;
    \draw[axis] (0,-5) -- (0,2) ;
    \draw[semithick] (1,-0.1) -- (1,0.1); 
    \draw[semithick] (2,-0.1) -- (2,0.1); 
    \draw[semithick] (3,-0.1) -- (3,0.1); 
    \draw[semithick] (4,-0.1) -- (4,0.1); 
    \draw[semithick] (5,-0.1) -- (5,0.1); 
    \draw[semithick] (6,-0.1) -- (6,0.1); 
    \draw[semithick] (7,-0.1) -- (7,0.1); 
    \draw[semithick] (8,-0.1) -- (8,0.1); 
    \draw[semithick] (9,-0.1) -- (9,0.1); 

    \draw[semithick] (-0.1,1) -- (0.1,1); 
    \draw[semithick] (-0.1,-1) -- (0.1,-1); 
    \draw[semithick] (-0.1,-2) -- (0.1,-2); 
    \draw[semithick] (-0.1,-3) -- (0.1,-3); 
    \draw[semithick] (-0.1,-4) -- (0.1,-4); 

    \draw [fill] (0,-1) circle [radius=.1] node [below left] (0,-1) {$(0,-1)$};
    \draw [fill] (2,-1) circle [radius=.1] node [right] (2,-1) {$(q-1,v(a_1))$};
    \draw [fill] (8,-3) circle [radius=.1] node [below] (8,-3) {$(q^2-1,v(a_2))$};

    \draw[thick] (0,-1) -- (8,-3); 
\end{tikzpicture} \hfill 
\begin{tikzpicture}[
    scale=0.5,
    axis/.style={thin, ->, >=stealth'},
    every node/.style={color=black},
    ]
    \tiny

    \draw[axis] (-1,0)  -- (10,0) ;
    \draw[axis] (0,-5) -- (0,2) ;
    \draw[semithick] (1,-0.1) -- (1,0.1); 
    \draw[semithick] (2,-0.1) -- (2,0.1); 
    \draw[semithick] (3,-0.1) -- (3,0.1); 
    \draw[semithick] (4,-0.1) -- (4,0.1); 
    \draw[semithick] (5,-0.1) -- (5,0.1); 
    \draw[semithick] (6,-0.1) -- (6,0.1); 
    \draw[semithick] (7,-0.1) -- (7,0.1); 
    \draw[semithick] (8,-0.1) -- (8,0.1); 
    \draw[semithick] (9,-0.1) -- (9,0.1); 

    \draw[semithick] (-0.1,1) -- (0.1,1); 
    \draw[semithick] (-0.1,-1) -- (0.1,-1); 
    \draw[semithick] (-0.1,-2) -- (0.1,-2); 
    \draw[semithick] (-0.1,-3) -- (0.1,-3); 
    \draw[semithick] (-0.1,-4) -- (0.1,-4); 

    \draw [fill] (1,-1) circle [radius=.1] node [below, xshift=-1mm ] (1,-1) {$(1,-1)$};
    \draw [fill] (3,-1) circle [radius=.1] node [right] (3,-1) {$(q,v(a_1))$};
    \draw [fill] (9,-3) circle [radius=.1] node [below] (9,-3) {$(q^2,v(a_2))$};
    \draw [fill] (0,0.15) circle [radius=.1] node [left, yshift=1.5mm] (0,0.15) {$(0,v(e_0))$};
    
    \draw[thick] (0,0.15) -- (1,-1);
    \draw[thick] (1,-1) -- (9,-3); 
    \draw[dotted, thick] (0,0.15) -- (1,0);    
\end{tikzpicture} 
\caption{Situation a): Newton polygons for $\phi_\theta(X)/X$ and $\phi_\theta(X)-e_0$.}\label{fig:1}
\end{figure}

In case a), the point $(q-1,v(a_1))$ lies above (or on) this line, hence the Newton polygon of the polynomial has exactly one segment, and this is of length $q^2-1$ and slope
$\frac{v(a_2)+1}{q^2-1}$. Hence, each non-zero torsion element $e_0$ has valuation
\[ v(e_0)=-\frac{v(a_2)+1}{q^2-1}. \]

For computing a $\theta^2$-torsion element above such an $e_0$, we have to consider the defining equation
$\phi_\theta(X)-e_0=a_2X^{q^2}+a_1X^q+\theta X-e_{0}$. Again the point $(q,v(a_1))$ lies above (or on) the line through $(1,v(\theta))=(1,-1)$ and $(q^2,v(a_2))$ which is given by the equation $y=-1+\frac{v(a_2)+1}{q^2-1}\cdot (x-1)$. The segment through the points $(0,v(e_0))$ and $(1,-1)$ has slope $-1-v(e_0)$, and therefore, its slope is smaller than the slope of the other segment which is 
$\frac{v(a_2)+1}{q^2-1}=-v(e_0)$ (see Figure \ref{fig:1}). Hence, both segments are sides of the Newton polygon, and in particular there is a segment of length $1$ and slope $-1-v(e_0)$. Hence there is a root $e_1$ in $\KI(e_0)$ which satisfies $v(e_1)=1+v(e_{0})$.

Inductively, one then sees that for $e_{i-1}\in \phi[\theta^{i}]$ with $v(e_{i-1})\geq -\frac{v(a_2)+1}{q^2-1}$, the defining polynomial $a_2X^{q^2}+a_1X^q+\theta X-e_{i-1}$ for the $\theta^{i+1}$-torsion elements lying above $e_{i-1}$ has a segment of length $1$ and slope $-1-v(e_{i-1})$, namely the one through the points $(0,v(e_{i-1}))$ and $(1,-1)$. This implies that there is a root $e_i$ in $\KI(e_{i-1})$ satisfying $v(e_i)=1+v(e_{i-1})$. Hence, indeed $\KI(\phi[\theta^{n+1}])\subseteq \KI(\phi[\theta])$ for all $n\geq 1$, and so $\KI(\Lambda_\phi)=\KI(\phi[\theta])$
by Proposition \ref{prop:lattice--torsion-points}.
Since $\KI(\phi[\theta])$ is the splitting field of the (separable) polynomial $\phi_\theta(X)$, it is a Galois extension of $\KI$. As the roots of $\phi_\theta(X)$ form an $\Fq$-vector space of dimension $2$, the Galois group is a subgroup of $\GL_2(\Fq)$, and we obtain the claim by recognizing that $\# \GL_2(\Fq) =(q^2-1)\cdot (q^2-q)$.

\medskip

\noindent
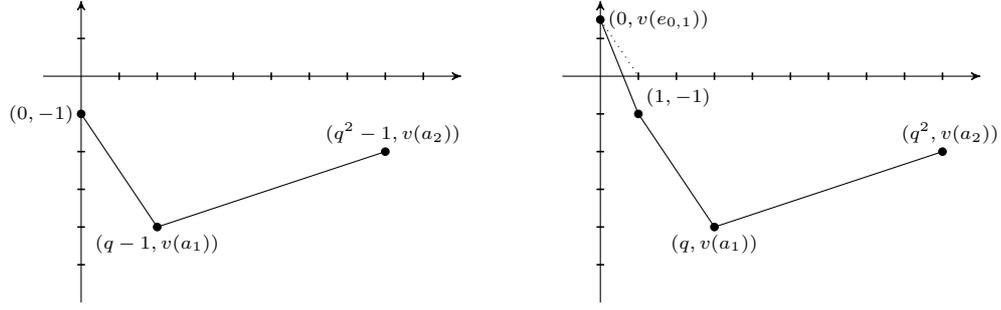
\begin{figure}[ht] 
\begin{tikzpicture}[
    scale=0.5,
    axis/.style={thin, ->, >=stealth'},
    every node/.style={color=black},
    ]
    \tiny

    \draw[axis] (-1,0)  -- (10,0) ;
    \draw[axis] (0,-6) -- (0,2) ;
    \draw[semithick] (1,-0.1) -- (1,0.1); 
    \draw[semithick] (2,-0.1) -- (2,0.1); 
    \draw[semithick] (3,-0.1) -- (3,0.1); 
    \draw[semithick] (4,-0.1) -- (4,0.1); 
    \draw[semithick] (5,-0.1) -- (5,0.1); 
    \draw[semithick] (6,-0.1) -- (6,0.1); 
    \draw[semithick] (7,-0.1) -- (7,0.1); 
    \draw[semithick] (8,-0.1) -- (8,0.1); 
    \draw[semithick] (9,-0.1) -- (9,0.1); 

    \draw[semithick] (-0.1,1) -- (0.1,1); 
    \draw[semithick] (-0.1,-1) -- (0.1,-1); 
    \draw[semithick] (-0.1,-2) -- (0.1,-2); 
    \draw[semithick] (-0.1,-3) -- (0.1,-3); 
    \draw[semithick] (-0.1,-4) -- (0.1,-4); 
    \draw[semithick] (-0.1,-5) -- (0.1,-5); 

	\coordinate (a0) at (0,-1);
	\coordinate (a1) at (2,-4);
	\coordinate (a2) at (8,-2);
	
    \draw [fill] (a0) circle [radius=.1] node [left] {$(0,-1)$};
    \draw [fill] (a1) circle [radius=.1] node [below]  {$(q-1,v(a_1))$};
    \draw [fill] (a2) circle [radius=.1] node [above,xshift=1mm]  {$(q^2-1,v(a_2))$};

    \draw (a0) -- (a1);
    \draw (a1) -- (a2);
    
\end{tikzpicture}\hfill 
\begin{tikzpicture}[
    scale=0.5,
    axis/.style={thin, ->, >=stealth'},
    every node/.style={color=black},
    ]
    \tiny

    \draw[axis] (-1,0)  -- (10,0) ;
    \draw[axis] (0,-6) -- (0,2) ;
    \draw[semithick] (1,-0.1) -- (1,0.1); 
    \draw[semithick] (2,-0.1) -- (2,0.1); 
    \draw[semithick] (3,-0.1) -- (3,0.1); 
    \draw[semithick] (4,-0.1) -- (4,0.1); 
    \draw[semithick] (5,-0.1) -- (5,0.1); 
    \draw[semithick] (6,-0.1) -- (6,0.1); 
    \draw[semithick] (7,-0.1) -- (7,0.1); 
    \draw[semithick] (8,-0.1) -- (8,0.1); 
    \draw[semithick] (9,-0.1) -- (9,0.1); 

    \draw[semithick] (-0.1,1) -- (0.1,1); 
    \draw[semithick] (-0.1,-1) -- (0.1,-1); 
    \draw[semithick] (-0.1,-2) -- (0.1,-2); 
    \draw[semithick] (-0.1,-3) -- (0.1,-3); 
    \draw[semithick] (-0.1,-4) -- (0.1,-4); 
    \draw[semithick] (-0.1,-5) -- (0.1,-5); 

	\coordinate (O) at (1,0);
	\coordinate (a0) at (1,-1);
	\coordinate (a1) at (3,-4);
	\coordinate (a2) at (9,-2);
	\coordinate (e01) at (0,1.5);
	
    \draw [fill] (a0) circle [radius=.1] node [above right] {$(1,-1)$};
    \draw [fill] (a1) circle [radius=.1] node [below]  {$(q,v(a_1))$};
    \draw [fill] (a2) circle [radius=.1] node [above,xshift=1mm]  {$(q^2,v(a_2))$};
    \draw [fill] (e01) circle [radius=.1] node [right]  {$(0,v(e_{0,1}))$};

    \draw (a0) -- (a1);
    \draw (a1) -- (a2);
    \draw (e01) -- (a0);
	\draw[dotted] (e01) -- (O);
\end{tikzpicture}
\caption{Situation b): Newton polygons for $\phi_\theta(X)/X$, and for $\phi_\theta(X)-e_0$ in the first case.}\label{fig:2}
\end{figure}

In case b), the Newton polygon for the polynomial $a_2X^{q^2-1}+a_1X^{q-1}+\theta $ for the $\theta$-torsion has two segments, since the point $(q-1,v(a_1))$ lies below the line through $(0,v(\theta))=(0,-1)$ and $(q^2-1,v(a_2))$. One segment is of length $q-1$ and slope $\frac{v(a_1)+1}{q-1}$, and the other of length $q^2-q$ and slope $\frac{v(a_2)-v(a_1)}{q^2-q}$ (see Figure \ref{fig:2}).

For $e_0$ a $\theta$-torsion point with $v(e_0)=-\frac{v(a_1)+1}{q-1}$, as in case a) one obtains a converging division tower $e_1, e_2, \ldots$ with $e_i\in \KI(e_0)$, and $v(e_i)=v(e_{i-1})+1$ (comp.~Figure \ref{fig:2}).

Now, consider a $\theta$-torsion point $e_0$ with $v(e_0)=-\frac{v(a_2)-v(a_1)}{q^2-q}=
\frac{v(a_1)-v(a_2)}{q(q-1)}$. For computing a $\theta^2$-torsion element above such an $e_0$, we again have to consider the defining equation
$a_2X^{q^2}+a_1X^q+\theta X-e_{0}$, and its Newton polygon determined by the four points $(0,v(e_0))$, $(1,-1)$, $(q,v(a_1))$ and $(q^2,v(a_2))$ (see Figure \ref{fig:3}). The first two points lie above the line through the last two points, hence one segment of the Newton polygon is the one through the latter points.
\noindent
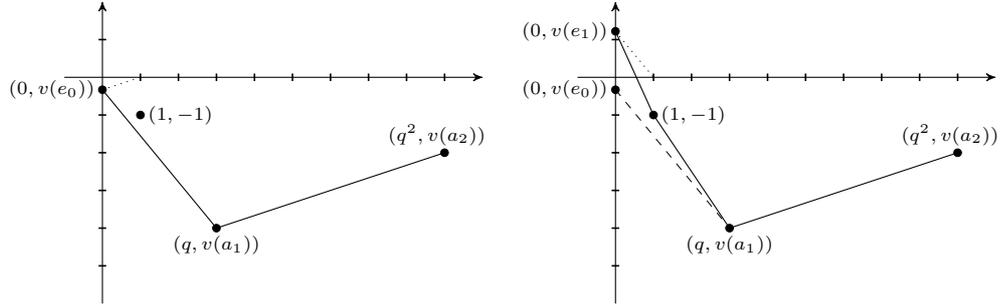
\begin{figure}[ht] 
\begin{tikzpicture}[
    scale=0.5,
    axis/.style={thin, ->, >=stealth'},
    every node/.style={color=black},
    ]
    \tiny

    \draw[axis] (-1,0)  -- (10,0) ;
    \draw[axis] (0,-6) -- (0,2) ;
    \draw[semithick] (1,-0.1) -- (1,0.1); 
    \draw[semithick] (2,-0.1) -- (2,0.1); 
    \draw[semithick] (3,-0.1) -- (3,0.1); 
    \draw[semithick] (4,-0.1) -- (4,0.1); 
    \draw[semithick] (5,-0.1) -- (5,0.1); 
    \draw[semithick] (6,-0.1) -- (6,0.1); 
    \draw[semithick] (7,-0.1) -- (7,0.1); 
    \draw[semithick] (8,-0.1) -- (8,0.1); 
    \draw[semithick] (9,-0.1) -- (9,0.1); 

    \draw[semithick] (-0.1,1) -- (0.1,1); 
    \draw[semithick] (-0.1,-1) -- (0.1,-1); 
    \draw[semithick] (-0.1,-2) -- (0.1,-2); 
    \draw[semithick] (-0.1,-3) -- (0.1,-3); 
    \draw[semithick] (-0.1,-4) -- (0.1,-4); 
    \draw[semithick] (-0.1,-5) -- (0.1,-5); 

	\coordinate (O) at (1,0);
	\coordinate (a0) at (1,-1);
	\coordinate (a1) at (3,-4);
	\coordinate (a2) at (9,-2);
	\coordinate (e02) at (0,-0.33);
	
    \draw [fill] (a0) circle [radius=.1] node [right] {$(1,-1)$};
    \draw [fill] (a1) circle [radius=.1] node [below]  {$(q,v(a_1))$};
    \draw [fill] (a2) circle [radius=.1] node [above,xshift=-1mm]  {$(q^2,v(a_2))$};
    \draw [fill] (e02) circle [radius=.1] node [left]  {$(0,v(e_{0}))$};

    \draw (a1) -- (a2);
	\draw[dotted] (e02) -- (O);
	\draw (e02) -- (a1);
\end{tikzpicture}\hfill 
\begin{tikzpicture}[
    scale=0.5,
    axis/.style={thin, ->, >=stealth'},
    every node/.style={color=black},
    ]
    \tiny

    \draw[axis] (-1,0)  -- (10,0) ;
    \draw[axis] (0,-6) -- (0,2) ;
    \draw[semithick] (1,-0.1) -- (1,0.1); 
    \draw[semithick] (2,-0.1) -- (2,0.1); 
    \draw[semithick] (3,-0.1) -- (3,0.1); 
    \draw[semithick] (4,-0.1) -- (4,0.1); 
    \draw[semithick] (5,-0.1) -- (5,0.1); 
    \draw[semithick] (6,-0.1) -- (6,0.1); 
    \draw[semithick] (7,-0.1) -- (7,0.1); 
    \draw[semithick] (8,-0.1) -- (8,0.1); 
    \draw[semithick] (9,-0.1) -- (9,0.1); 

    \draw[semithick] (-0.1,1) -- (0.1,1); 
    \draw[semithick] (-0.1,-1) -- (0.1,-1); 
    \draw[semithick] (-0.1,-2) -- (0.1,-2); 
    \draw[semithick] (-0.1,-3) -- (0.1,-3); 
    \draw[semithick] (-0.1,-4) -- (0.1,-4); 
    \draw[semithick] (-0.1,-5) -- (0.1,-5); 

	\coordinate (O) at (1,0);
	\coordinate (a0) at (1,-1);
	\coordinate (a1) at (3,-4);
	\coordinate (a2) at (9,-2);
	\coordinate (e02) at (0,-0.33);
	\coordinate (e12) at (0,1.22);
	
    \draw [fill] (a0) circle [radius=.1] node [right] {$(1,-1)$};
    \draw [fill] (a1) circle [radius=.1] node [below]  {$(q,v(a_1))$};
    \draw [fill] (a2) circle [radius=.1] node [above,xshift=-1mm]  {$(q^2,v(a_2))$};
    \draw [fill] (e02) circle [radius=.1] node [left]  {$(0,v(e_{0}))$};
    \draw [fill] (e12) circle [radius=.1] node [left]  {$(0,v(e_{1}))$};

    \draw (a1) -- (a2);
	\draw (e12) -- (a0);
	\draw (a0) -- (a1);
	\draw[very thin, dashed] (e02) -- (a1);
	\draw[dotted] (e12) -- (O);
	
\end{tikzpicture}
\caption{Situation b), second case: Newton polygons for $\phi_\theta(X)-e_0$, and for $\phi_\theta(X)-e_1$.}\label{fig:3}
\end{figure}

If $(0,v(e_0))$ lies below (or on) the line through $(1,-1)$ and $(q,v(a_1))$, we obtain the segment through  $(0,v(e_0))$ and $(q,v(a_1))$ as a side of the Newton polygon (as in Figure \ref{fig:3}). In this case, we obtain a $\theta^2$-torsion point $e_1$ above $e_0$ of valuation
\[  v(e_1)=-\frac{v(a_1)-v(e_0)}{q}=\frac{v(e_0)-v(a_1)}{q}=
\frac{v(a_1)-v(a_2)}{q^2(q-1)}-\frac{v(a_1)}{q} \]
(compare Figure \ref{fig:3}).

This continues inductively for $e_i$ ($i\geq 2$), as long as the point $(0,v(e_{i-1}))$ is below or on the line through $(1,-1)$ and $(q,v(a_1))$, and we obtain in this case

\begin{eqnarray*}
v(e_i)&=&\frac{v(e_{i-1})-v(a_1)}{q}=\frac{v(a_1)-v(a_2)}{q^{i+1}(q-1)}-v(a_1)\cdot \left(\frac{1}{q}+\frac{1}{q^2}+\ldots +\frac{1}{q^i}\right)\\
&=&\frac{v(a_1)-v(a_2)}{q^{i+1}(q-1)}-v(a_1)\cdot \frac{1-(1/q)^i}{q(1-1/q)}\\
&=&\frac{v(a_1)(1+q-q^{i+1})-v(a_2)}{q^{i+1}(q-1)}.
\end{eqnarray*}

After that the Newton polygon always has a segment of length $1$ from $(0,v(e_{i-1}))$ to $(1,-1)$ (as already in Figure \ref{fig:3} for $i-1=1$). Hence, $e_i\in \KI(e_{i-1})$ with $v(e_i)=v(e_{i-1})+1$.

The point $(0,v(e_{i-1}))$ is below the line through $(1,-1)$ and $(q,v(a_1))$, if $-1-e_{i-1}>\frac{v(a_1)+1}{q-1}$, i.e., if
\begin{eqnarray*}
 -1-\frac{v(a_1)(1+q-q^{i})-v(a_2)}{q^{i}(q-1)} &>& \frac{v(a_1)+1}{q-1} \\
\Leftrightarrow \qquad v(a_1) &<& \frac{v(a_2)-q^{i+1}}{q+1}.
\end{eqnarray*}

Therefore, by definition of $n$, this holds for $i\leq n$. Hence, $\KI(\Lambda_\phi)=
\KI(\phi[\theta],e_1,\ldots, e_n)=\KI(\phi[\theta^{n+1}])$.

From the construction and the Newton polygons, we see that $[\KI(\phi[\theta]):\KI]$ divides $(q^2-q)(q-1)=q(q-1)^2$,
and $[\KI(\phi[\theta^{n+1}]):\KI(\phi[\theta])]$ is a divisor of $q^{n}$. This gives the upper bound.
From the valuation of $e_n$, we see that in case that $v(a_1)-v(a_2)$ is prime to $q$, we have a ramification of order at least $q^{n+1}$, leading to the lower bound.
\end{proof}

\begin{rem}
We never used that the Drinfeld module is defined over $K$. We only used the valuations of the defining coefficients. Therefore, the theorem holds for any field inside $\KI$. Even more, for any finite extension $L_\infty$ of $\KI$, the same proof works apart from the explicit lower bound in b). This is the only point where we used that the valuation has integer values at elements in the base field. 
Nevertheless, appropriate choices of $a_1$ and $a_2$ still lead to arbitrarily large extensions of $L_\infty$.
\end{rem}

\begin{rem}
Even for Drinfeld modules $\phi$ with everywhere good reduction the extension $\KI(\Lambda_\phi)/\KI$ can be arbitrarily large. Indeed, we can choose any $a_2\in \Fq^\times$, and just have to arrange $a_1$  in such a way that
$ v(a_1)<\frac{-q^{n+1}}{q+1}$
 for a given $n\in \NN$. Then part b) of Theorem \ref{thm:bounds} shows that $q^{n+1}$ divides the degree of the extension. 
\end{rem}

\bibliographystyle{alpha}
\def\cprime{$'$}

\vspace*{.5cm}

\parindent0cm

\end{document}